\documentclass[a4paper,draft]{amsproc}
\usepackage{amssymb}
\usepackage[hyphens]{url} 
\usepackage[all]{xy}
\usepackage{amsthm}
\usepackage{amscd}

\theoremstyle{plain}
\newtheorem{theorem}{Theorem}[section]
\newtheorem{lem}[theorem]{Lemma}
\newtheorem{cor}[theorem]{Corollary}
\newtheorem{prop}[theorem]{Proposition}
\newtheorem{remark}[theorem]{Remark}

\theoremstyle{definition}
\newtheorem{defn}[theorem]{Definition}

\theoremstyle{remark}
\newtheorem{question}{Question}

\newcommand{\Q}{\mathbb{Q}}
\newcommand{\Z}{\mathbb{Z}}

\newcommand{\G}{\mathcal{G}}

\newtheorem{innercustomgeneric}{\customgenericname}
\providecommand{\customgenericname}{}
\newcommand{\newcustomtheorem}[2]{%
	\newenvironment{#1}[1]
	{%
		\renewcommand\customgenericname{#2}%
		\renewcommand\theinnercustomgeneric{##1}%
		\innercustomgeneric
	}
	{\endinnercustomgeneric}
}

\newcustomtheorem{customthm}{Theorem}

\renewcommand{\leq}{\leqslant}
\renewcommand{\geq}{\geqslant}
\renewcommand{\setminus}{\smallsetminus}
\title[Profinite genus]{Profinite genus of fundamental groups of compact flat manifolds with holonomy group of prime order}


\author[G. J. Nery]{\bfseries Genildo de Jesus Nery} 

\address{ 
Department of Mathematics \\ 
University of Brasilia   \\ 
Brasilia\\
Brazil}
\email{g.j.nery@mat.unb.br}

\thanks{The author held CNPq scholarship during the preparation of this article.} 

\begin{document}

{\begin{flushleft}\baselineskip9pt\scriptsize
\end{flushleft}}
\vspace{18mm} \setcounter{page}{1} \thispagestyle{empty}

\begin{abstract}
In this article we calculate the profinite genus of the fundamental group  of a $n$-dimensional compact flat manifold $X$ with holonomy group of prime order. As consequence, we prove that if $n\leq 21$, then $X$ is determined among all $n$-dimensional compact flat manifolds by the profinite completion of its fundamental group. Furthermore, we characterize the isomorphism class of profinite completion of the fundamental group of $X$ in terms of the representation genus of its holonomy group.
\end{abstract}

\maketitle

\section{Introduction}  

		There has been much recent study of whether residually finite groups, or classes of residually finite groups connected with geometry or topology may be distinguished from each other by their sets of finite quotient groups. 
		
		In group theory the study in this direction started in 70-th of the last centure when Baumslag [Bau74],  Stebe [Ste72] and others found examples of finitely generated residually finite groups having the same set of finite quotients. The general question  addressed in this study can be formulated as follows: 
		
		\begin{question}\label{question}
			To what extent a finitely generated residually finite group $\Gamma$ is  determined by its profinite completion?
		\end{question} 
		
		The study leaded to the notion of genus $\mathfrak{g}(\Gamma)$, the set of isomorphism classes of finitely generated residually finite groups having the same set of finite quotients as $\Gamma$. Equivalently, $\mathfrak{g}(\Gamma)$ is the set of isomorphism classes of finitely generated residually finite groups having the same profinite completion isomorphic to the profinite completion $\widehat \Gamma$ of $\Gamma$. In fact, the term genus was  borrowed from integral representation theory, where for a finite group $G$  the genus of a $\Z G$-lattice $M$ is defined as the set of isomorphism classes of $\Z G$-lattices $N$ such that the $\widehat \Z G$-modules $\widehat M$ and $\widehat N$ are isomorphic.    
		
		The study usually was concentrated to establish whether the cardinality $|\mathfrak{g}(\Gamma)|$ of   genus  is finite or  1 ( see \cite{GPS} or \cite{GS} for example). 
		
		The recent stream of research was concentrated on important properties of manifolds that can be detected by the profinite completion (see \cite{WZ-10, WZ-16, Wil19}) or the \textquoteleft rigidity\textquoteright \ results showing that the fundamental group of certain manifolds has genus 1, i.e., determined by their profinite completions (see \cite{BCR16, BMRS18, BMRS20, Wil17}).  For the recent account of many results of these studies we refer the reader to \cite{Reid18}. There are only few papers where exact numbers or estimates of the genus appear due to the difficulties of such calculation (see \cite{GZ11, BZ, Ner19}). 
		
		\bigskip
		We concentrate here on  $n$-dimensional compact flat manifolds. It is well known that these manifolds are well described by the famous Bieberbach’s theorems (see \cite{Cha86}). Consequently, the fundamental group $\pi_{1}(X)$ of $X$ is called Bieberbach group. Such groups are characterized as follows: a $n$-dimensional Bieberbach group  $\pi_{1}(X)$ is a torsion-free group having  a finitely generated maximal abelian torsion-free subgroup $M$ of $\pi_{1}(X)$ of finite index (see \cite{AK57}). The groups $M$ and $G=\pi_1(X)/M$ are known as the translation subgroup  and the holonomy group of $\pi_{1}(X)$ (or of $X$), respectively. A well known Auslander-Kuranishi theorem \cite{AK57} says that every finite group is the holonomy group of a compact flat manifold. Furthermore, a compact flat manifold is uniquely determined by its fundamental group, up to diffeomorphism.
		
		It  was proved recently in \cite{PPW} that all  Bieberbach groups of dimension $\leq 4$ are determined by their profinite completions.
		The main result of this  paper is the explicit  formula  for the genus of a $n$-dimensional Bieberbach group with holonomy group of prime order. Denote by $\Q(\zeta)$ the field of rational numbers with the $p$-th root of unity attached, by   $H(\Q(\zeta))$ its class group and by $\G(p)=G(\Q(\zeta)/\Q)$ its Galois group, the latter acts naturally on $H(\Q(\zeta))$.
		
		\begin{theorem}\label{main Theorem}
			Let  $\Gamma$ be a $n$-dimensional Bieberbach group whose holonomy group has prime order. 

\begin{enumerate}

\item[(i)]			If all indecomposable summands of $\Z G$-module $M$ have $\Z$-rank $p-1$ except one trivial summand of $\Z$-rank $1$,  then $|\mathfrak{g}(\Gamma)|=|C_2\backslash H(\Q(\zeta))|$, where $C_2$ is the group of order $2$ acting on $H(\Q(\zeta))$ by invertion.

\item[(ii)] Otherwise $|\mathfrak{g}(\Gamma)|=|\mathcal{G}(p)\backslash H(\Q(\zeta))|$.
\end{enumerate} 	
		\end{theorem} 
		
		As a consequence of this result we obtain the following statement.
		
		\begin{cor}\label{cor-1}
			Let $\Gamma$ be a $n$-dimensional Bieberbach group with holonomy group $G$ of prime order. Then 
			
			$|\mathfrak{g}(\Gamma)|=1$ if and only if $|G|\leq 19$.
		\end{cor}
		
		In other words, Corollary \ref{cor-1} asserts that a $n$-dimensional compact flat manifold $X$ with holonomy group $G$ of prime order is determined among all $n$-dimensional compact flat manifolds by $\widehat{\pi_{1}(X)}$ if and only if $|G|\leq 19$. Furthermore, using the crystallographic restriction theorem (see \cite{Hil85}) we obtain the following statement. 
		\begin{cor}\label{cor-2}
			Let $X$ be a $n$-dimensional compact flat manifold with holonomy group of prime order. If  $n\leq 21$, then $X$ is determined among all $n$-dimensional compact flat manifolds by the profinite completion of its fundamental group.
		\end{cor}
		
		The proof of Theorem \ref{main Theorem} uses the Charlap's classification of Bieberbach groups with holonomy group of prime order (see Subsection \ref{section-Bieberbach groups}). Furthermore, a crucial fact is the following correspondence that we establish  showing   that the genus $\mathfrak{g}(\Gamma)$ in our case coicides with the genus of the $\Z G$-lattice $M$ (in the represantation theory sense) in case (ii) of Theorem \ref{main Theorem}. 
		
		\begin{theorem}\label{main Theorem 1} Let $M$ be a free abelian group of rank $n$ and $G$ a group of order $p$. 
			Let $\Gamma_1$ and $\Gamma_2$ be $n$-dimensional Bieberbach groups which are extensions of $M$ by $G$. If  $M_1$ and $M_2$ are $G$-modules induced by the action of $\Gamma_1$ and $\Gamma_2$ on $M$, respectively, then 
			
			\begin{enumerate}
				
				\item[(i)]    $\Gamma_1\cong \Gamma_2$  if and only if either $G$-modules $M_1$ and $M_2$ satisfy hypothesis (ii) of Theorem \ref{main Theorem} and are isomorphic or $M_1$ and $M_2$ satisfy hypothesis (i) of Theorem \ref{main Theorem} and  isomorphic up to a twist of $G$ by inversion. 
				\item[(ii)]			
				$\widehat{\Gamma}_1\cong\widehat{\Gamma}_2$  if and only if  $\widehat{M}_1\cong\widehat{M}_2$ as $G$-modules.
			\end{enumerate}
		\end{theorem}

		Throughout this paper, we use the following notations and terminology:  $\widehat{\Gamma}$ denotes the profinite completion of a group $\Gamma$ and $|S|$ denotes the cardinality of a set $S$, $p$ will be a prime number, $\zeta$ a primitive $p$-th root of unity, $H(\Q(\zeta))$ the ideal class group of the cyclotomic field $\Q(\zeta)$ with ring of integers $\Z[\zeta]$, $h_p$ the order of $H(\Q(\zeta))$, $\mathcal{G}(p)$ the Galois group of the field $\Q(\zeta)$ over $\Q$ and $G$ a cyclic group of order $p$ generated by an element $x$. 	
		
		\section{Preliminaries}	 
		
		\subsection{Modules over groups of prime order}\label{C_p-modules}
		
		\begin{defn}  Let $H$ be a finite group.
			A  $\Z H$-lattice is a  $\Z H$-module which is free as a $\Z$-module.
		\end{defn}
		
		According to \cite[Theorem 74.3]{CR62} there are exactly $2h_p+1$ indecomposable $ \Z G$-lattices (up to isomorphism). Namely denoting by $A_i \ (1\leq i\leq h_p)$   representatives of the $h_p$ distinct ideal classes of $\Q(\zeta)$, the indecomposable $ \Z G$-lattices are:
		\begin{enumerate}
			\item[(a)] the trivial module $\Z$.
			\item[(b)]  $A_i$  with action of $G=\langle x\rangle$ given by $x\cdotp a:=\zeta a$, $a\in A_i$. 
			\item[(c)] $h_p$ modules $(A_i,a_i):=A_i\oplus \Z$, where  $a_i\in A_i\setminus (\zeta -1)A_i$ and the action of $G$ is defined by $x\cdotp (a,r):=(\zeta a+ ra_i,r)$, for $a\in A_i$ and $r\in \Z$.
		\end{enumerate}
		
		\begin{remark}\label{(A,a)=ZC}
			If $p$ does not divide $n \ (n\in \Z)$, then $(\Z[\zeta],n)\cong \Z G$ is a cyclic free $\Z G$-module (see \cite[p. 512 and p. 514]{CR62}).
		\end{remark}
		
		Thus we have
		
		\begin{prop}[\cite{CR62}, Theorem 74.3]\label{Diederichsen-Reiner}
			Every finitely generated  $\Z G$-lattice $M$ is isomorphic to a direct sum
			\begin{equation}\label{sum 1}
			M=\bigoplus_{i=1}^{c}(A_i,a_i)\oplus\bigoplus_{j=1}^{b}A_{j}\oplus \Z^{a}
			\end{equation}
			where $\Z^{a}$ is a trivial $\Z G$-module of rank $a$. 
		\end{prop}
		
		\begin{defn}\label{exceptional}
			We will say that a $\Z G$-module $M$ is exceptional if in decomposition \eqref{sum 1} we have the invariants $a=1$ and $c=0$, i.e., if $M=\bigoplus_{j=1}^{b}A_{j}\oplus \Z$.
		\end{defn}
		
		Since $\mathbf{Z}_p G$ is local ring, for $\mathbf{Z}_p G$-modules we have the Krull-Schmidt-Azumaya's Theorem.
		
		\begin{prop}[\cite{CR81}, Theorem 6.12]\label{KSA}
			Let $R$ is a complete commutative noetherian local ring and $M$ be a finitely generated  $R$-module.
			Then $M$ is  a  direct sum of indecomposable submodules. Furthermore, if
			\begin{equation*}
			M=\bigoplus_{i=1}^{r}M_{i}=\bigoplus_{j=1}^{s}N_{j}
			\end{equation*}
			are two such sums, then $r=s$ and $M_1\cong N_{j_1},\cdots, M_r\cong N_{j_r}$, where $\{j_1,\cdots,j_r\}$ is some permutation of $\{1,\cdots,r\}$.
		\end{prop}
		
		There are exactly three non-isomorphic indecomposable
		$\mathbf{Z}_pG$-modules.
		
		\begin{prop}[\cite{HR62}, Theorem 2.6]\label{Heller-Reiner}
			The only indecomposable $\mathbf{Z}_p G$-modules (up to
			isomorphism) are   $\mathbf{Z}_p$, $\mathbf{Z}_p[\zeta]$ and $\mathbf{Z}_pG$.
		\end{prop}

		Let $M$ be a $\Z G$-lattice. Let us denote by $M_p=M\otimes_{\Z} \mathbf{Z}_p$ the $\mathbf{Z}_p G$-lattice. Note that $M_p$ is the pro-$p$ component of $\widehat{M}$.
		
		\begin{prop}[\cite{HR62}, Corollary 1.5]\label{HR62-Cor 1.5}
			The $\Z G$-module $M$ is indecomposable if and only if the corresponding $\mathbf{Z}_p G$-module $M_p$  is indecomposable.
		\end{prop}
		
		\begin{remark}\label{Remark-2}
			\begin{enumerate}
				\item[(i)]  The pro-$p$ completion of each non-zero ideal $A$ of $\Z[\zeta]$ is isomorphic to $\mathbf{Z}_p[\zeta]$ as  $\mathbf{Z}_pG$-modules, because $\mathbf{Z}_p[\zeta]$ is a local principal ideal domain.
				\item[(ii)] The pro-$p$ completion of $(A,a_0)$ is isomorphic to $\mathbf{Z}_pG$ as $\mathbf{Z}_pG$-modules, where $A$ is a  non-zero ideal of $\Z[\zeta]$. Indeed, this follows from Propositions \ref{Heller-Reiner} and \ref{HR62-Cor 1.5}. 	
			\end{enumerate}	 
		\end{remark}
		
		\begin{prop}[\cite{CR81}, Proposition 31.2 (ii)]\label{31.2 (ii)}
			Let $G$ be a finite group and $M$ and $N$ be $\Z G$-lattices. Then $M_q\cong N_q$ as $\mathbf{Z}_q G$-modules for all primes $q$   if and only if $M_p\cong N_p$ as $\mathbf{Z}_p G$-modules for all $p$ dividing $|G|$.
		\end{prop}
		
		\subsection{Galois groups acting on ideal class groups}
		
		The Galois group $\mathcal{G}(p)$ naturally acts on $\Q(\zeta)$ (respectively $\Z[\zeta]$) via automorphisms. In particular, $\mathcal{G}(p)$ acts on the set of ideals of $\Z[\zeta]$.
		
		\begin{lem}[\cite{Cha86}, Chapter IV, Exercise 6.2]\label{Exercise 6.2}
			Let $A$ and $B$ be ideals of $\Z[\zeta]$. If $A$ and $B$ are in the same ideal class, then $\sigma(A)$ and $\sigma(B)$ are in the same ideal class for any $\sigma \in \mathcal{G}(p)$.
		\end{lem}
		
		It follows from Lemma \ref{Exercise 6.2} that $\mathcal{G}(p)$ acts via automorphisms on $H(\Q(\zeta))$. We will denote by $\mathcal{G}(p)\backslash H(\Q(\zeta))$  the set of orbits of $H(\Q(\zeta))$ under this action.
		
		The cyclotomic fields of class number one were characterized by Montgomery, as well as  Uchida in 1971.
		
		\begin{prop}[\cite{Rib01}, p. 652]\label{h_p=1 iff p<19}
			Let $p$ be a prime number. Then $h_p=1$ if and only if $p\leq 19$.
		\end{prop}
		
		\begin{defn}
			Let $M$ and $N$ be $\Z G$-modules. A semi-linear homomorphism from $M$ to $N$  is a pair $(f,\varphi)$  where $f:M\rightarrow N$ is an abelian group homomorphism and $\varphi$ is  an automorphism of $G$ such that $$f(x\cdotp m)=\varphi(x)\cdotp f(m)$$ for $x\in G$ and $m\in M$.
		\end{defn}
		
		Let $[A]$ denote the ideal class of an ideal $A$ of $\Z[\zeta]$ and let $M(a,b,c;[A])$ be the unique (up to isomorphism) $\Z G$-lattice with invariants $a, \ b, \ c$ and $[A]$ (see \cite[Theorem 74.3]{CR62}).
		
		\begin{prop}[\cite{Cha86}, Chapter IV, Theorem 6.2]\label{Cha-th 6.2}
			Let $M=M(a,b,c;[A])$ and $M'=M(a',b',c';[A'])$ be $\Z G$-lattices. Then $M$ will be semi-linearly isomorphic to $M'$ if and only if $a=a', \ b=b', c=c'$ and $\sigma\cdotp [A]=[A']$ for some $\sigma\in \mathcal{G}(p)$.
		\end{prop}

		\subsection{Bieberbach groups with prime order holonomy}\label{section-Bieberbach groups}

		This subsection contains the classification of Bieberbach groups whose holonomy group has prime order. This classification is due to Charlap \cite{Cha65}.
		
		\begin{prop}[\cite{Cha86}, Chapter I, Proposition 4.1]\label{Cha, Prop 4.1}
			Let $\Gamma$ be a $n$-dimensional Bieberbach group. Then the translation subgroup $M$ is the unique normal, maximal abelian subgroup of $\Gamma$.
		\end{prop}
		
		Using the exact sequence 
		\begin{equation}\label{exact sequence}
		1\rightarrow M\rightarrow \pi_{1}(X)\rightarrow G\rightarrow 1,
		\end{equation}  one sees that $M\cong \Z^n$ has a natural structure of $\Z G$-module. This defines the representation $\rho: G\rightarrow \mathrm{GL}(n,\Z)$ of the holonomy group $G$. Since $M$ is the maximal abelian subgroup, it follows that $\rho$ is faithful (i.e., $\rho$ is injective). It is well known that a representation $\rho$ induces a structure of $\Z G$-module on $M$.
		
		\begin{lem}\label{characterization-Bieberbach group with n(G)=1}
			Let $\Gamma$ be a $n$-dimensional Bieberbach group with holonomy group $G$ and  $M$ its maximal abelian normal subgroup. Then $M=M_{n-1}\oplus \Z$ admits a $\Z G$-decomposition, where $\Z$ is trivial module generated by $p$-th power of some element $c$ of $\Gamma$ and $\Gamma=M_{n-1}\rtimes C$ with $C=\langle c\rangle$.
		\end{lem}
		\begin{proof}
			By \cite[Theorem 2]{Szc97} we have the following exact sequence
			\begin{equation*}
			1\rightarrow M_{n-1}\rightarrow \Gamma\rightarrow \Z \rightarrow 1.
			\end{equation*} 
			Since $\Z$ is free, this sequence splits as a semidirect product $\Gamma=M_{n-1}\rtimes C$. The lemma is proved.
		\end{proof}
		\begin{defn}
			A Bieberbach group $\Gamma$ with prime order holonomy group $G$ is exceptional if its maximal abelian normal subgroup $M$  is an exceptional $\Z G$-module.
		\end{defn}
		
		\begin{prop}[\cite{Cha86}, Chapter IV, Theorem 6.3]\label{C1-Bieberbach}
			There is a one-to-one correspondence between isomorphism classes of non-exceptional Bieberbach groups whose holonomy group has prime order $p$ and $4$-tuples $(a,b,c; \theta)$ where $a,b,c\in \Z$ are as in \eqref{sum 1} with $a> 0, \ b\geq 0, \ c\geq 0, \ (a,c)\neq (1,0), \ (b,c)\neq (0,0)$ and $\theta \in\mathcal{G}(p)\backslash H(\Q(\zeta))$. 
		\end{prop}
		
		Note that $\mathcal{G}(p)$ is cyclic of order $p-1$ if $p>2$, and of order 2 if $p=2$. Hence, $\mathcal{G}(p)$ has a unique subgroup $C_2$ of order $2$. 
		
		\begin{prop}[\cite{Cha86}, Chapter IV, Theorem 6.4]\label{C2-Bieberbach}
			There is a one-to-one correspondence between isomorphism classes of exceptional Bieberbach groups whose holonomy group has prime order $p$ and pairs $(b,\theta)$ where $b>0$ as in Definition \ref{exceptional},  and $\theta \in C_2\backslash H(\Q(\zeta))$. 
		\end{prop}

\begin{remark}\label{charlap remark} By \cite[Remark p. 151]{Cha86} the isomorphism class of a Bieberbach group $\Gamma$ determines the structure of  $\Z G$-module $M$ in not exceptional case and  determines it up to a twist by inversion of $G$ in the exceptional case. 

\end{remark}	
	
		We shall need the following

		\begin{prop}[\cite{Hil85}, Theorem 1.5]\label{crystallographic restriction} Let $\Gamma$ be a Bieberbach group with holonomy group of order $p$. Then the rank of maximal free abelian normal subgroup $n$ of $\Gamma$ is at least $p-1$. 
		\end{prop}

		\section{Profinite genus of Bieberbach groups with prime order holonomy}\label{genus}

		\begin{lem}\label{unique open}
			Let $\Gamma$ be a $n$-dimensional Bieberbach group with translation subgroup $M$. Then $\widehat{M}$ is the unique open normal, torsion-free maximal abelian subgroup of $\widehat{\Gamma}$.
		\end{lem}
		
		\begin{proof} There  is a one-to-one correspondence between the set of all finite index normal subgroups of  $\Gamma$  and the set of all open normal subgroups of $\widehat{\Gamma}$ (see \cite[Proposition 3.2.2]{RZ00}). Therefore the lemma follows from Proposition \ref{Cha, Prop 4.1}. 
		\end{proof}
		
		From (\ref{exact sequence}) and Lemma \ref{unique open} we obtain the exact sequence
		\begin{equation*}
		1\rightarrow\widehat{M}\rightarrow\widehat{\Gamma}\rightarrow G\rightarrow 1
		\end{equation*}
		for the corresponding profinite completions.
		
		\begin{lem}\label{diagram}
			Let $\Gamma_1$ and $\Gamma_2$ be $n$-dimensional Bieberbach groups with translation subgroups $M_1$ and $M_2$ and  holonomy groups $G_1$ and $G_2$, respectively. If $\varphi:\widehat{\Gamma}_1\rightarrow\widehat{\Gamma}_2$ is an isomorphism, then there are isomorphisms $\phi:\widehat{M}_1\rightarrow \widehat{M}_2$ and $\psi: G_1\rightarrow G_2$ such that the following diagram commutes:
			\begin{equation}\label{diagram 1}
			\begin{CD}
			1 @>{}>> \widehat{M}_1 @>{}>> \widehat{\Gamma}_1 @>{}>> G_1 @>>> 1 \\
			@. @VV{\phi}V @VV{\varphi}V @VV{\psi}V\\
			1 @>{}>> \widehat{M}_2 @>{}>> \widehat{\Gamma}_{2} @>{}>> G_2 @>>> 1.
			\end{CD}
			\end{equation}
		\end{lem}
		\begin{proof}
			By Lemma \ref{unique open}, $\varphi(\widehat{M}_1)=\widehat{M}_2$. We define $\phi$ to be the restriction  of $\varphi$ to $\widehat{M}_1$. Since $\widehat{M}_1, \widehat{M}_2$ are profinite free abelian groups of the same rank,   $\phi: \widehat{M}_1\rightarrow\widehat{M}_2$ is an isomorphism. Thus $\varphi$ induces an isomorphism $\psi: G_1\rightarrow G_2$ such that the diagram (\ref{diagram 1}) commutes.
		\end{proof}
		
		Note that $\widehat M$ can be considered simply as $G$-modules (see  \cite[Proposition 5.3.6 (c)]{RZ00}).
		
		\begin{lem}\label{condition to isomorphism}
			Let $$M=\bigoplus_{i=1}^{c}(A_i,a_i)\oplus\bigoplus_{j=1}^{b}A_{j}\oplus \Z^{a},\ \ \  M'=\bigoplus_{i=1}^{c'}(B_i,b_i)\oplus\bigoplus_{j=1}^{b'}B_{j}\oplus\Z^{a'}$$ be $\Z G$-lattices. Then  $\widehat{M}\cong \widehat{M'}$ as $G$-modules if and only if $a=a', b=b'$ and $c=c'$.
		\end{lem}
		\begin{proof} 
			Our proof starts with the observation that
			\begin{equation*}
			\widehat{M}= \bigoplus_{i=1}^{c}\widehat{(A_i,a_i)}\oplus\bigoplus_{j=1}^{b}\widehat{A}_{j}\oplus\widehat{\Z}^{a} \  \ \text{and} \ \ \widehat M'=\bigoplus_{i=1}^{c'}\widehat{(B_i,b_i)}\oplus\bigoplus_{j=1}^{b'}\widehat B_{j}\oplus\widehat{\Z}^{a'}.
			\end{equation*}
			If $\widehat{M}\cong\widehat{M'}$ as $G$-modules, then the pro-$p$ components of $\widehat{M}$ and $\widehat{M'}$ are isomorphic as $\mathbf{Z}_pG$-modules. Therefore, the necessity follows from Remark \ref{Remark-2} together with Propositions \ref{KSA} and \ref{Heller-Reiner}.
			
			On the other hand, by Remark \ref{Remark-2} the statement holds for pro-$p$ components of $\widehat{M}$ and $\widehat{M'}$. Hence by Proposition \ref{31.2 (ii)} $M_q\cong N_q$ for all primes $q$ and so $\widehat{M}\cong\widehat{M'}$ as $G$-modules.		
		\end{proof}

		Let $\psi$ be an automorphism of a finite group $G$ and let $M$ be a $G$-module. Let $(M)^{\psi}$ denote the $G$-module $M$ given by the action $g\cdotp m:=\psi(g)\cdotp m$ for $g\in G$ and $m\in M$.  
		
		\begin{prop}\label{item (i) th 1.4}
		Let $\Gamma_1$ and $\Gamma_2$ be $n$-dimensional Bieberbach groups with  holonomy group $G$ and translation subgroups $M_1$ and $M_2$, respectively.
		Then $\Gamma_1\cong \Gamma_2$ if and only if $M_1\cong (M_2)^{\psi}$ as $G$-modules where $\psi \in  \mathrm{Aut}(G)$.
		\end{prop}
		\begin{proof}
			  Suppose  $\varphi: \Gamma_1\rightarrow\Gamma_2$ is an isomorphism. By Proposition \ref{Cha, Prop 4.1}, $M_1$ and $M_2$ are the unique normal, maximal abelian subgroups of $\Gamma_1$ and $\Gamma_2$, respectively. Hence, $\varphi(M_1)=M_2$. Therefore, $\varphi$  induces an isomorphism $\psi:\Gamma_1/M_1\rightarrow\Gamma_2/M_2$, i.e., $\psi \in \mathrm{Aut}(G)$ because $\Gamma_1/M_1 \cong G\cong \Gamma_2/M_2$. Define $\phi:= \varphi|_{M_1}$. By definition of $\phi$ and $\psi$ the following diagram commutes
				\begin{equation*}
				\begin{CD}
				1 @>{}>> M_1 @>{}>> \Gamma_1 @>{\pi_1}>> G @>>> 1 \\
				@. @VV{\phi}V @VV{\varphi}V @VV{\psi}V\\
				1 @>{}>> M_2 @>{}>> \Gamma_{2} @>{\pi_2}>> G @>>> 1.
				\end{CD}
				\end{equation*}
				Thus,
				\begin{equation*}
				\phi(x\cdotp m)=\varphi(\alpha m \alpha^{-1})=\varphi(\alpha)\varphi(m)\varphi(\alpha)^{-1}=\psi(x)\cdotp \phi(m)
				\end{equation*} 
				for $x\in G, m\in M_1$ and $\alpha \in \Gamma_1$ with $\pi_1(\alpha)=x$. Therefore, $M_1\cong (M_2)^{\psi}$ as $G$-modules.  
				
				Conversely, suppose there is $\psi\in \mathrm{Aut}(G)$ such that $M_1\cong (M_2)^{\psi}$ as $G$-modules. Then $M_1$ is semi-linearly isomorphic to $M_2$. Therefore, by Propositions \ref{Cha-th 6.2}, \ref{C1-Bieberbach} and \ref{C2-Bieberbach} we have $\Gamma_1\cong\Gamma_2$. 		
		\end{proof}
			
		\begin{proof}[Proof of Theorem \ref{main Theorem 1}] 
			
			\begin{enumerate}
				\item[(i)] It follows from Proposition \ref{item (i) th 1.4} and  by Remark \ref{charlap remark}.
				
				\item[(ii)] \textquoteleft Only if\textquoteright.
				From Lemma \ref{diagram} we see that there are isomorphisms $\phi:\widehat{M}_1\rightarrow \widehat{M}_2$ and $\psi: G\rightarrow G$ such that the following diagram is commutative
				\begin{equation*}
				\begin{CD}
				1 @>{}>> \widehat{M}_1 @>{}>> \widehat{\Gamma}_1 @>{}>> G @>>> 1 \\
				@. @VV{\phi}V @VV{\varphi}V @VV{\psi}V\\
				1 @>{}>> \widehat{M}_2 @>{}>> \widehat{\Gamma}_{2} @>{}>> G @>>> 1.
				\end{CD}
				\end{equation*}
				Then $\phi:\widehat{M}_1\rightarrow \widehat{M}_2$ is a $G$-module homomorphism which is an isomorphism by Lemma \ref{diagram}. 
				
				Conversely, by Lemma \ref{characterization-Bieberbach group with n(G)=1}, $M=M_{n-1}\oplus \Z$, $M'=M'_{n-1}\oplus \Z'$ such that 
				$\Gamma_1=M_{n-1}\rtimes C_1$ and  $\Gamma_2={M'}_{n-1}\rtimes C_2$, where $C_1,C_2$ contain $\Z$ and $\Z'$ as subgroups of index $p$ and act on $M_{n-1}, {M'}_{n-1}$ as $G$.
				Hence, $\widehat M=\widehat M_{n-1}\oplus \widehat\Z$, $\widehat M'=\widehat M'_{n-1}\oplus \widehat \Z'$ and 
				$\widehat \Gamma_1=\widehat M_{n-1}\rtimes \widehat C_1$ and  $\widehat \Gamma_2=\widehat M'_{n-1}\rtimes \widehat C_2$, where $\widehat C_1,\widehat C_2$ act on $\widehat M_{n-1}, \widehat M'_{n-1}$ as $G$. 	
				By Lemma \ref{condition to isomorphism} and Proposition \ref{KSA}, we get that the profinite completions of $G$-modules $\widehat M_{n-1}$, $\widehat M'_{n-1}$ are isomorphic (since they are isomorphic as $\mathbf{Z}_pG$-modules). Hence, $\widehat \Gamma_1$ is isomorphic to $\widehat \Gamma_2$. 
			\end{enumerate}
		\end{proof}
		
		The profinite version of Propositions \ref{C1-Bieberbach} and \ref{C2-Bieberbach} is the following.

		\begin{cor}\label{cor-3}
			Let  $\Gamma$ be a $n$-dimensional Bieberbach group whose holonomy group has prime order. Then there is a one-to-one correspondence between isomorphism classes of $\widehat{\Gamma}$ and triples $(a,b,c)$ from the decomposition \eqref{sum 1} with $a>0$. 
		\end{cor}
		
		Now we are ready to prove Theorem \ref{main Theorem}.

		\begin{proof}[Proof of Theorem \ref{main Theorem}]
			Let $\Gamma_1$ and $\Gamma_2$ be $n$-dimensional Bieberbach groups with holonomy group of prime order and translation subgroups $M_1$ and $M_2$, respectively. By Lemma \ref{diagram} we can assume that $\Gamma_1$ and $\Gamma_2$ have the same holonomy group $G$ of prime order $p$. As $M_1$ and $M_2$ are $\Z G$-modules, by Proposition \ref{Diederichsen-Reiner},
			\begin{equation*}
			M\cong \bigoplus_{i=1}^{c}(A_i,a_i)\oplus\bigoplus_{j=1}^{b}A_{j}\oplus\Z^{a}
			\end{equation*}
			and 
			\begin{equation*}
			M'\cong \bigoplus_{i=1}^{c'}(B_i,b_i)\oplus\bigoplus_{j=1}^{b'}B_{j}\oplus \Z^{a'}
			\end{equation*}
			where  $A_i, \  B_i$ are ideals of $\Z[\zeta]$. Since $\widehat{\Gamma}_1\cong\widehat{\Gamma}_2$, we have $\widehat{M}_1\cong\widehat{M}_2$ as $G$-modules (see Theorem \ref{main Theorem 1}). Hence, $a=a', \ b=b'$ and $c=c'$ by Lemma \ref{condition to isomorphism}. Therefore, the statement follows from Propositions \ref{C1-Bieberbach} and \ref{C2-Bieberbach} together with Theorem \ref{main Theorem 1}.
		\end{proof}
		
		\begin{cor}\label{example 1} 
			Let $\Gamma$ be a  Bieberbach group whose holonomy group has prime order. If $H(\Q(\zeta))$ is non-trivial, then   $|\mathfrak{g}(\Gamma)|>1$.
		\end{cor}

\begin{remark} It follows  that the cardinality $|\mathfrak{g}(\Gamma)|$ of the genus of a Bieberbach group $\Gamma$ with the translation group $M$ and holonomy group $G$ of order $p$ equals to the cardinality of the genus of the $G$-module $M$ for non-exceptional $\Gamma$ and half of it if $\Gamma$ is exceptional.

\end{remark}

		\begin{proof}[Proof of Corollary \ref{cor-1}]
			This is a consequence of Theorem \ref{main Theorem} together with Proposition \ref{h_p=1 iff p<19}.
		\end{proof}

		\begin{proof}[Proof of Corollary \ref{cor-2}]
			From Proposition \ref{crystallographic restriction}, we deduce that for all $n\leq 21$ no prime number $p\geq 23$ can occur as orders of elements in the holonomy groups of $n$-dimensional Bieberbach groups. Then the result  follows from Corollary \ref{cor-1}.
		\end{proof}
		
		\begin{remark} If $\Gamma=M\rtimes G$, with $M$ free abelian of rank $n$ and $G$ of order $p$. Then it is not a Biebierbach group. Nevertheless using that $\Gamma$ and $M_1\rtimes G$ are isomorphic if and only if $M\cong M_1$ as $G$-modules and that the same holds for the profinite completions one can deduce that the genus of $\Gamma$ has the same cardinality  as in Theorem \ref{main Theorem}. Moreover, considering $G$ as a subgroup of $\mathrm{Aut}(M)=\mathrm{GL}(n, \Z)$ one deduces from \cite[Proposition 2.17]{GZ11} that the cardinality of the genus is exactly the number of conjugacy classes of subgroups of order $p$ of $\mathrm{GL}(n, \Z)$ in the conjugacy class of $G$ in $\mathrm{GL}(n,\widehat\Z)$. 	
		\end{remark} 
		
		\section*{Acknowledgements}
		The author wishes to express his thanks to Prof. Dr. Pavel Zalesskii for many discussions and advices.

\bibliographystyle{amsplain}

\end{document}